\documentclass[12pt,a4paper,reqno]{amsart}
\usepackage{amsmath}
\usepackage{amsfonts}
\usepackage{amssymb}
\usepackage{bm} 

\newcommand\R{{\mathbf{R}}}

\newcommand\M{{\operatorname{M}}}
\newcommand\E{{\operatorname{E}}}
\renewcommand\Im{{\operatorname{Im}}}

\newcommand\eps{{\varepsilon}}

\newcommand\dist{{\operatorname{dist}}}

\newcommand{\nabb}{\mbox{$\nabla \mkern-13mu /$\,}}

\theoremstyle{plain}
  \newtheorem{theorem}[subsection]{Theorem}
  
  \newtheorem{proposition}[subsection]{Proposition}
  \newtheorem{lemma}[subsection]{Lemma}

\theoremstyle{remark}
  \newtheorem{remark}[subsection]{Remark}

\theoremstyle{definition}

\include{psfig}

\begin{document}

\title[A global compact attractor for NLS with potential]{A global compact attractor for high-dimensional defocusing non-linear Schr\"odinger equations with potential}
\author{Terence Tao}
\address{Department of Mathematics, UCLA, Los Angeles CA 90095-1555}
\email{tao@@math.ucla.edu}

\vspace{-0.3in}
\begin{abstract}
We study the asymptotic behavior of large data solutions in the energy space $H := H^1(\R^d)$ in very high dimension $d \geq 11$ to defocusing Schr\"odinger equations $i u_t + \Delta u = |u|^{p-1} u + Vu$ in $\R^d$, where $V \in C^\infty_0(\R^d)$ is a real potential  (which could contain bound states), and $1+\frac{4}{d} < p < 1+\frac{4}{d-2}$ is an exponent which is energy-subcritical and mass-supercritical.  In the spherically symmetric case, we show that as $t \to +\infty$, these solutions split into a radiation term that evolves according to the linear Schr\"odinger equation, and a remainder which converges in $H$ to a compact attractor $K$, which consists of the union of spherically symmetric almost periodic orbits of the NLS flow in $H$.  The main novelty of this result is that $K$ is a \emph{global} attractor, being independent of the initial energy of the initial data; in particular, no matter how large the initial data is, all but a bounded amount of energy is radiated away in the limit.
\end{abstract}

\maketitle

\section{Introduction}

The purpose of this paper is to establish some asymptotic properties of bounded-energy solutions of non-linear Schr\"odinger (NLS) equations 
\begin{equation}\label{nls}
i u_t + \Delta u = |u|^{p-1} u + Vu
\end{equation}
with high dimension $d \geq 5$, where the potential $V \in C^\infty_0(\R^d)$ is real, and where the defocusing nonlinearity $|u|^{p-1} u$ is energy-subcritical and mass-supercritical, which means that $1 + \frac{4}{d} < p < 1 + \frac{4}{d-2}$.  No size, positivity, or spectral assumptions will be made on the potential $V$, which may thus be arbitrarily large, negative, and contain many bound states.  However, as is well known, the high dimension $d \geq 5$ does automatically exclude any resonances for the linear Schr\"odinger operator $-\Delta + V$ at the origin; see \cite{Jensen}.  For technical reasons we will unfortunately be forced to place further assumptions on the dimension, and ultimately our main results will only hold for $d \geq 11$, though one can improve this bound somewhat with more effort (see Section \ref{remarks-sec}).

The equation \eqref{nls} enjoys two conserved quantities, the \emph{mass}
\begin{equation}\label{mass-def}
\M(u) = \M(u(t)) := \int_{\R^d} |u(t,x)|^2\ dx
\end{equation}
and the \emph{energy}
\begin{equation}\label{energy-def}
\E(u) = \E(u(t)) := \int_{\R^d} \frac{1}{2} |\nabla u(t,x)|^2 + V(x) \frac{1}{2} |u(t,x)|^2 + \frac{1}{p+1} |u(t,x)|^{p+1}\ dx.
\end{equation}
It is known (see e.g. \cite{cazbooknew}, \cite{katounique}, \cite{tao:book}) that for any initial data $u_0$ in the energy space\footnote{This is the Hilbert space with inner product $\langle u, v \rangle_H := \int_{\R^d} u \overline{v} + \nabla u \cdot \overline{\nabla v}\ dx$.} $H := H^1_x(\R^d)$, there exists a unique local strong solution $u \in C^0_t H^1_x([-T,T] \times \R^d)$ to \eqref{nls} with that data which conserves both the mass and energy, where we recall that a \emph{strong solution} is a solution in $C^0_t H^1_x$ for which the Duhamel formula
$$ u(t) = e^{it\Delta} u_0 - i \int_0^t e^{i(t-t')\Delta} (|u|^{p-1} u + Vu)(t')\ dt'$$
is valid in the sense of distributions.  Furthermore, the time $T$ of existence depends only on the dimension $d$, the potential $V$, and the $H^1$ norm of $u_0$.  From the conservation laws and the Gagliardo-Nirenberg inequality we obtain the \emph{a priori} bound
$$ \| u(t) \|_{H} \lesssim_{p,d,V,\|u_0\|_{H}} 1,$$
where we use $X \lesssim_k Y$ or $X = O_k(Y)$ to denote the estimate $X \leq C(k) Y$ for some constant $C$ depending only on $k$, and similarly for other choices of subscripts.  One can then easily iterate the local existence theory and conclude that there is a unique \emph{global} strong solution $u \in C^0_t H^1_x(\R \times \R^d)$ to \eqref{nls} from any initial data $u_0 \in H^1_x(\R^d)$.

Now we consider the long-term behaviour of such global solutions as $t \to \pm \infty$.  By the time reversal symmetry $u(t,x) \mapsto \overline{u(-t,x)}$ it suffices to consider the limit $t \to +\infty$.  A major tool for this task is provided by the \emph{generalised virial identity}
\begin{equation}\label{virial}
\begin{split}
\partial_t \int_{\R^d} \nabla a \cdot \Im( \overline{u} \nabla u )\ dx
& = 2 \int_{\R^d} \operatorname{Hess}(a)( \nabla u, \overline{\nabla u} )\ dx\\
&\quad + \frac{p-1}{p+1} \int_{\R^d} |u|^{p+1} \Delta a\ dx \\
&\quad - \frac{1}{2} \int_{\R^d} |u|^2 \Delta \Delta a\ dx \\
&\quad - \int_{\R^d} (\nabla a \cdot \nabla V) |u|^2\ dx
\end{split}
\end{equation}
for any test function $a \in C^\infty_0(\R^d)$, where $\operatorname{Hess}(a)$ is the Hessian quadratic form
$$ \operatorname{Hess}(a)(v,w) := \sum_{i=1}^d \sum_{j=1}^d \frac{\partial^2 a}{\partial x_i x_j} v_i w_j.$$
This identity can be easily verified by formal computation, and can be justified rigorously by regularising the nonlinearity $|u|^{p-1} u$ and the initial data $u_0$; we omit the standard details.  One can also extend the identity to more general classes of weights $a$ assuming sufficient regularity and decay conditions on $u$; see Section \ref{virial-sec}.

Suppose that we are in the free case $V=0$.  Formally applying \eqref{virial} with $a(x) := |x|$ and then integrating in time, we obtain the \emph{Morawetz inequality}
$$ \int_0^\infty \int_{\R^d} \frac{|\nabb u|^2}{|x|} + \frac{|u|^{p+1}}{|x|} + \frac{|u|^2}{|x|^3}\ dx dt \lesssim_{p,d,\|u_0\|_{H}} 1$$
where $|\nabb u|^2 := |\nabla u|^2 - |\frac{x}{|x|} \cdot \nabla u|^2$ is the angular component of the energy density $|\nabla u|^2$.  This Morawetz inequality can be used to justify a \emph{scattering} result (or more precisely, an \emph{asymptotic completeness} result): given any $u_0 \in H$, there exists a unique \emph{scattering state} $u_+ \in H$ such that $\| u(t) - e^{it\Delta} u_+ \|_{H} \to 0$ as $t \to +\infty$; see \cite{gv:scatter}.  In other words, we have an asymptotic $u(t) = e^{it\Delta} u_+ + o_H(1)$, where $o_H(1)$ denotes a function which goes to zero in $H$ norm as $t \to \infty$.

Similar arguments can be made when $V$ is radially decreasing or is sufficiently small.  However, when the negative component of $V$ is large, the linear Schr\"odinger $-\Delta + V$ can admit bound states, which then implies the existence of nonlinear bound states $Q$ that solve the equation $-EQ + \Delta Q = |Q|^{p-1} Q + V Q$ for some fixed energy $E \in \R$; see \cite{rose} for further discussion.  This leads to solutions $u(t,x) := e^{iEt} Q(x)$ to \eqref{nls} which do not converge to a free solution $e^{it\Delta} u_+$, and so asymptotic completeness in the formulation given above fails for such potentials.

However, one may hope that such nonlinear bound states are the \emph{only} obstruction to asymptotic completeness for the equation \eqref{nls}, and more specifically that any global solution $u \in C^0_t H^1_x([0,+\infty) \times \R^d)$ should asymptotically take the form $u(t) = e^{it\Delta} u_+ + e^{iEt} Q + o_H(1)$ for some nonlinear bound state $Q$.  Such an assertion would be consistent with the (somewhat imprecise) \emph{soliton resolution conjecture}; see \cite{soffer-icm}, \cite{tao:compact}, \cite{attractor} for further discussion.

Such a precise asymptotic appears to be well out of reach of current technology at present\footnote{An exception to this occurs when the nonlinearity is restricted to only a finite number of points in space, thus reducing the system to a linear dispersive PDE coupled with a nonlinear ODE; see \cite{komech}.}.  However we are able to present the following partial result under the additional assumption of spherical symmetry (i.e. $u(t,x) = u(t,|x|)$) and very high dimension $d \geq 11$, which is the main result of this paper:

\begin{theorem}[Global compact attractor]\label{main-thm} Fix $p, d, V$ as above.  Suppose also that $d \ge 11$. Then there exists a compact set $K \subset H$ of spherically symmetric functions, invariant under the flow \eqref{nls}, such that for every global solution $u \in C^0_t H^1_x(\R \times \R^d)$ which is spherically symmetric, there exists a spherically symmetric scattering state $u_+ \in H$ such that
$$ \dist_H( u(t) - e^{it\Delta} u_+, K ) \to 0 \hbox{ as } t \to +\infty.$$
\end{theorem}

Thus we have a resolution of the form $u(t) = e^{it\Delta} u_+ + w(t) + o_H(1)$, where $w(t)$ ranges inside a universal compact subset of $H$.  The condition $d \geq 11$ can be relaxed; see Section \ref{remarks-sec}.

The arguments in \cite{attractor} already yield a weaker version of this theorem, in which the solution $u$ is assumed to have bounded energy (and the compact set $K$ is then allowed to depend on this energy bound); we review these arguments in Section \ref{quasisec}.  The main novelty in this paper, therefore, is the fact that $K$ is now independent of the energy of the initial data; in other words, $K$ is a truly \emph{global} attractor for the evolution (once one subtracts away the radiation term, of course).  In particular, we see that $\limsup_{t \to +\infty} \| u(t) - e^{it\Delta} u_+ \|_H$ remains bounded even as the energy of $u$ goes to infinity; to put it another way, every finite energy solution, no matter how large, radiates all but a bounded amount of its energy to infinity.

Theorem \ref{main-thm} raises the possibility that the soliton resolution conjecture could in principle be establishable for \eqref{nls} for specific choices of $p, d, V$ by means of rigorous numerics, combined with a quantitative nonlinear stability analysis of each of the nonlinear bound states.  Indeed, suppose one knew that each nonlinear bound state was orbitally stable (cf. \cite{rose}), even in the presence\footnote{Note that in some models, radiation is known to introduce instabilities in an otherwise stable system; see \cite{radiation}.} of a radiation term $e^{it\Delta} u_+$, provided that the time $t$ was sufficiently large (in order to allow the radiation term to decay) and that the remaining portion of the solution lay within $\eps$ (say) of the nonlinear bound state in $H$ norm.  Suppose also the attractor $K$ was known to be contained in some other compact set $K'$, and that one could show by rigorous numerics (using quantitative perturbative analysis to control errors) that any initial data in $K'$ would eventually end up either exiting $K'$, or coming within $\eps$ of a nonlinear bound state, after a bounded amount of time.  Then one could conclude that any finite energy solution would eventually decouple into a radiation term plus a term which always stayed within a small distance of a set of nonlinear bound states.  If one then had some asymptotic stability results for such bound states one could thus (in principle) establish the soliton resolution conjecture for this model \eqref{nls}; examples of such stability results (in the small energy regime) can be found in \cite{sw}, \cite{sw2}, \cite{wayne}, \cite{weder}.  Note however that such stability results are only likely to be true for the ground nonlinear bound states; excited bound states are likely to decay into bound states of lower energy; see \cite{sw3} for an instance of this.  In principle, though, a modified stability analysis of such excited states as in \cite{sw3} could still suffice, in conjunction with rigorous numerics, to establish the soliton resolution conjecture for any given model of the form \eqref{nls}.

Our argument proceeds as follows.  First, we adapt the arguments in \cite{attractor} to obtain a preliminary compactness result which allows one to reduce matters to establishing the result for \emph{almost periodic} spherically symmetric solutions, i.e. spherically symmetric solutions whose orbit $\{ u(t): t \in [0,+\infty)\}$ is precompact in $H$.  For such solutions we establish additional spatial decay properties, using the double Duhamel trick from \cite{attractor}.  The decay properties will allow us to utilise the generalised virial identity \eqref{virial} for functions $a$ which grow rather fast at infinity to yield universal bounds on the mass and energy of such solutions.  More precisely, we use the classical virial weight $a(x) := |x|^2$ to control the energy, and the quartic weight $a(x) := |x|^4$ to control the mass.  In order to use the latter weight one requires quite strong spatial decay on the almost periodic solution, which is why our results are restricted to high dimension $d \geq 11$.  Invoking the arguments from \cite{attractor} once more then gives the desired compactness to the space of almost periodic solutions.

\subsection{Acknowledgments}

The author is supported by a grant from the Macarthur Foundation and by NSF grant DMS-0649473.  The author thanks Michael Weinstein for posing this question.  The author also thanks the referee for corrections.

\subsection{Notation}

Throughout the remainder of this paper, $p,d,V$ are understood to be fixed and to obey the hypotheses above (i.e. that $d \geq 5$, $1+\frac{4}{d} < p < 1+\frac{4}{d-2}$, and $V \in C^\infty_0(\R^d)$).  All implied constants are allowed to depend on $p,d,V$.

\section{Reduction to a quasi-Liouville theorem}\label{quasisec}

The first step in the argument is to establish a local compact attractor:

\begin{theorem}[Local compact attractor]\label{local-thm} Let $0 < E < \infty$.  Then there exists a compact set $K_E \subset H$ of spherically symmetric functions, invariant under the flow \eqref{nls}, such that for every global solution $u \in C^0_t H^1_x(\R \times \R^d)$ which is spherically symmetric and which obeys the bound $\sup_t \|u(t)\|_H \leq E$, there exists a spherically symmetric scattering state $u_+ \in H$ such that
$$ \dist_H( u(t) - e^{it\Delta} u_+, K_E ) \to 0 \hbox{ as } t \to +\infty.$$
\end{theorem}

This result was already established in \cite[Theorem 1.12]{attractor}, when the potential $V$ was absent and the nonlinearity $|u|^{p-1} u$ was replaced by a more general nonlinearity $F(u)$ of $p^{th}$ power type.  It turns out that the presence of the lower order term $Vu$ in the equation \eqref{nls} makes essentially no difference to the arguments in \cite{attractor}, basically because $Vu$ obeys all\footnote{In the model case in which $V = G(Q)$ for a nonlinear bound state $Q$ to some NLS $-E Q + \Delta Q = F(Q)$ and some power type nonlinearities $F, G$ of order $p$ and $p-1$ respectively, one can in fact view \eqref{nls} as a vector-valued free nonlinear Schr\"odinger equation for the pair $(u, Q e^{iEt})$.  An inspection of the arguments in \cite{attractor} show that the fact that $u$ is scalar is never actually used in the paper, and so the results in \cite{attractor} extend without difficulty to this vector-valued setting, thus giving Theorem \ref{local-thm} in this case.  Since this model case already gives quite a large and ``generic'' class of potentials $V$ for Theorem \ref{local-thm}, this already gives a fairly convincing heuristic argument that Theorem \ref{local-thm} must hold in general.} the estimates that were required of $|u|^{p-1} u$.  Thus Theorem \ref{local-thm} can be established by repeating the arguments from \cite{attractor} \emph{mutatis mutandis}.  We will however provide a little more detail below concerning the (minor) changes in \cite{attractor} that have to be made to accomodate the potential.

Roughly speaking, the idea is to repeat the arguments in \cite{attractor} but with $F(u)$ replaced by $|u|^{p-1} u + Vu$ throughout.  The local bound $\|F(u)\|_{W^{1,R}_x(\R^d)} \lesssim \|u\|_H^p$ in \cite[Lemma 2.3]{attractor} holds as long as one adds the lower order term $\|u\|_H$ to the right-hand side (which turns out to be quite harmless).  The local Strichartz control in \cite[Lemma 4.3]{attractor} is then easily extended to the case of a potential by an application of H\"older's inequality.  The smoothing effect in \cite[Proposition 4.5]{attractor} for the nonlinearity $|u|^{p-1} u + Vu$ still holds, by exploiting the well-known Kato smoothing effect \cite{sj}, \cite{vega} to obtain additional regularity for $Vu$. In the proof of \cite[Lemma 5.6]{attractor}, the expression $F(v+e^{it\Delta}u_+)-F(v)$ that needs to be estimated there acquires an additional term of $V e^{it\Delta} u_+$, but this is easily seen to be manageable by H\"older's inequality.

With these estimates in hand, most of the rest of the argument in \cite{attractor} goes through with no changes, as the properties of the nonlinearity $F$ are only exploited through the above lemmas.  In particular, the preliminary decomposition of the solution in \cite[Proposition 5.2]{attractor} and the frequency localisation in \cite[Proposition 6.1]{attractor} remain true.  The preliminary spatial localisation in \cite[Theorem 7.1]{attractor} also can be established by repeating the arguments, with only one small modification: in addition to the concentration points $x_1(t),\ldots,x_J(t)$ identified in that proposition, one should add one additional concentration point at the origin.  This ensures that $V$ is small away from these concentration points, which allows the effect of $V$ to be neglected when applying a spatial cutoff $\chi$ away from these points.  (In any case, we will reprove this spatial localisation shortly.)  The arguments in \cite[Section 8]{attractor} then yield Theorem \ref{local-thm} with no changes (indeed, the nonlinearity is not even mentioned in that section).

\begin{remark} The above arguments did not use the defocusing nature of the nonlinearity, and would indeed hold for more general equations of the form $i u_t + \Delta u = F(u) + Vu$, where $F$ was as in \cite{attractor}.
\end{remark}

In view of Theorem \ref{local-thm}, the proof of Theorem \ref{main-thm} reduces to the following ``quasi-Liouville theorem'', which asserts that the space of almost periodic spherically symmetric solutions is compact:

\begin{theorem}[Quasi-Liouville theorem]\label{quasi-thm} Suppose that $d \geq 11$.  Then there exists a compact set $K \subset H$ such that any global solution $u \in C^0_t H^1_x(\R \times \R^d)$ which is spherically symmetric and which is \emph{almost periodic} in the sense that $\{ u(t): t \in \R \}$ is a precompact subset of $H$, must lie in $K$.
\end{theorem}

Indeed, to prove Theorem \ref{main-thm}, let $u$ be any global spherically symmetric solution to \eqref{nls} of finite energy. Then the conservation laws give $\sup_t \|u(t)\|_H \leq E$ for some finite $E$, and so by Theorem \ref{local-thm} we have a spherically symmetric scattering state $u_+$ such that $u - e^{it\Delta} u_+$ is attracted to an invariant compact set $K_E$ of spherically symmetric functions.  But every element of $K_E$ generates a solution which stays in $K_E$ (by invariance) and is thus almost periodic (by compactness), and so by Theorem \ref{quasi-thm} we thus have $K_E \subset K$ for some compact set $K$ independent of $E$, and Theorem \ref{main-thm} follows.

\begin{remark} In view of Theorem \ref{local-thm}, the soliton resolution conjecture can also be cast in an equivalent ``Liouville theorem'' form, asserting that the only almost periodic solutions are the soliton solutions.  See \cite{soffer-icm}, \cite{attractor} for some further discussion of this point.  The reduction of problems concerning general solutions to dispersive equations to that of establishing Liouville-type theorems for almost periodic solutions to such equations (or more generally, to solutions which are almost periodic modulo the symmetries of the equation) is now well established in the literature, by the work of Martel-Merle, Merle-Raphael, Kenig-Merle, and others; see e.g. \cite{tao-survey} for a survey.
\end{remark}

\section{Polynomial spatial decay}

From compactness, one easily sees that an almost periodic solution $u$ must exhibit uniform spatial decay in the sense that $\sup_t \int_{|x| \geq R} |u(t,x)|^2 + |\nabla u(t,x)|^2\ dx \to 0$ as $R \to \infty$ (see \cite[Proposition B.1]{attractor} for a proof).  In fact this decay can be made polynomial in $R$ (answering a question in \cite[Remark 1.21]{attractor} in the spherically symmetric case); establishing this will be the purpose of this section.  Such enhanced decay is necessary for us to apply the generalised virial identity \eqref{virial} with weights $a$ which grow fairly rapidly at infinity.

More precisely, we shall show

\begin{proposition}[Polynomial spatial decay]\label{decay-prop} Let $u$ be an almost periodic spherically symmetric global solution with 
\begin{equation}\label{supt}
\sup_t \|u(t)\|_H \leq E.
\end{equation}
Then for any $R \geq 1$ and $t \in \R$ we have
\begin{equation}\label{xr}
 \int_{|x| \geq R} |u(t,x)|^2\ dx \lesssim_E R^{4-d}.
\end{equation}
\end{proposition}

Note that this proposition works for all dimensions $d \geq 5$; the need for the stronger condition $d \geq 11$ only arises when applying the virial identity in the next section.

\begin{remark}  The exponent $R^{4-d}$ is natural. Indeed, consider a nonlinear bound state $Q$, thus
$-EQ + \Delta Q = |Q|^{p-1} Q + V Q$ for some $E$ (not the same as the $E$ in the proposition).  Then we can write $Q = - (-\Delta + E)^{-1} ( |Q|^{p-1} Q + V Q )$.  As is well known, the fundamental solution of $(-\Delta + E)^{-1}$ is bounded by $O( 1/|x|^{2-d} )$ uniformly in $E$, and so one expects $Q$ to decay like $O( 1/|x|^{2-d} )$ as well, which is consistent with \eqref{xr}.
\end{remark}

\begin{proof} Let $u$ be as in this proposition.  Let $R_0 \geq 1$ be the first radius greater than or equal to $1$ such that $V(x) = 0$ for all $x \geq R_0/100$; thus $R_0 \lesssim 1$.
For any $\alpha \geq 0$, let $P(\alpha)$ denote the assertion that
\begin{equation}\label{xr2}
\int_{|x| \geq R} |u(t,x)|^2\ dx \lesssim_{E,\alpha} R^{-\alpha}
\end{equation}
for all $R \geq R_0$ and $t \in \R$.  The claim $P(0)$ follows immediately from \eqref{supt}.  We need to show that $P(d-4)$ is true.  It will suffice to show that the implication
\begin{equation}\label{impl}
P(\alpha) \implies P( \min( d-4, \alpha + \delta ) )
\end{equation}
holds for all $\alpha \geq 0$, where $\delta > 0$ is a constant depending only on $p, d$, since this will imply $P(d-4)$ from $P(0)$ after finitely many iterations of \eqref{impl}.

It remains to prove \eqref{impl}.  We thus let $\alpha \geq 0$ be such that \eqref{xr2} holds.  We allow all implied constants to depend on $E,\alpha$.  Note (from \eqref{supt}) that \eqref{xr2} implies the variant
\begin{equation}\label{xr3}
\int_{|x| \geq R} |u(t,x)|^2\ dx \lesssim (1+R)^{-\alpha}
\end{equation}
for all $R \geq 0$ and $t \in \R$.

Fix a smooth cutoff function $\eta: \R^d \to [0,1]$ which equals $1$ when $|x| \geq 1$ and vanishes when $|x| \leq 1/2$, and let $\eta_R(x) := \eta(x/R)$ and $\chi_R := \eta_R - \eta_{2R}$.  To show $P( \min( d-4, \alpha +\delta )$, it will suffice by geometric series to show that
\begin{equation}\label{chir}
\| \chi_R u(t) \|_2^2 \lesssim R^{4-d} + R^{-\alpha - \delta}
\end{equation}
for all $t \in \R$ and $R \geq R_0$, and some $\delta > 0$ depending only on $p$ and $d$, where we abbreviate $\| \|_{L^2_x(\R^d)}$ as $\| \|_2$.

Fix $R$.  By time translation symmetry we may take $t=0$.  Let $T > R^2$ be a large time (which will eventually be sent to infinity).  From Duhamel's formula one has
\begin{equation}\label{chr}
\chi_R u(0) = u_\pm + v_\pm + w_\pm
\end{equation}
for either choice of sign $\pm$, where
\begin{align*}
u_\pm &:= \chi_R e^{\mp iT\Delta}u(\pm T) \\
v_\pm &:= i \chi_R \int_0^{\pm T} e^{-it\Delta} (1-\eta_{R/10}) F(u(t))\ dt\\
w_\pm &:= i \chi_R \int_0^{\pm T} e^{-it\Delta} \eta_{R/10} F(u(t))\ dt
\end{align*}
and $F(u) := |u|^{p-1} u + Vu$.  

From \eqref{chr} we can expand the left-hand side of \eqref{chir} for $t=0$ as
$$
\langle u_+ + v_+ + w_+, u_- + v_- + w_- \rangle_{L^2}$$
which after some application of Cauchy-Schwarz can be expressed as
$$ \langle v_+, v_- \rangle_{L^2} + O( ( \| u_+ \|_{2} + \| w_+ \|_{2} + \| u_- \|_{2} + \| w_- \|_{2} ) ( \| u_+ \|_{2} + \| w_+ \|_{2} + \| u_- \|_{2} + \| w_- \|_{2} + \| \chi_R u(0) \|_{2} ) ).$$
Using the elementary inequality $ab \leq \eps a^2 + \frac{1}{\eps} b^2$ to absorb the $\|\chi_R u(0)\|_2$ term into the left-hand side of \eqref{chir}, we thus conclude that
\begin{equation}\label{charge}
\| \chi_R u(0) \|_{2}^2 \lesssim |\langle v_+, v_- \rangle_{L^2}| + \|u_+\|_{2}^2 + \|w_+\|_{2}^2 + \|u_-\|_{2}^2 + \|w_-\|_{2}^2.
\end{equation}

We now estimate each of the terms on the right-hand side of \eqref{charge}.  From the almost periodic nature of $u$ and the Riemann-Lebesgue lemma for the free Schr\"odinger evolution (see \cite[Lemma B.5]{attractor}) we have
\begin{equation}\label{upp}
\lim_{T \to +\infty} \| u_+ \|_{2}^2 + \|u_-\|_{2}^2 = 0.
\end{equation}
Now we estimate $\|w_+\|_{2}$.  By the triangle inequality, we can estimate this by the sum of
\begin{equation}\label{ok}
\| \chi_R \int_{R^2}^{T} e^{-it\Delta} \eta_{R/10} F(u(t))\ dt \|_{2}
\end{equation}
and
\begin{equation}\label{ok2}
\| \int_0^{R^2} e^{-it\Delta} \eta_{R/10} F(u(t))\ dt \|_{2}.
\end{equation}
Let us first consider \eqref{ok}.  By duality, we can express \eqref{ok} as
$$ \int_{R^2}^T \langle \eta_{R/10} F(u(t)), e^{it\Delta} \chi_R f \rangle_{L^2}\ dt$$
for some $f$ with $\|f\|_{2}=1$.  Applying the standard dispersive inequality
\begin{equation}\label{eitf}
\| e^{it\Delta} f \|_{L^r_x(\R^d)} \lesssim |t|^{-d(\frac{1}{2}-\frac{1}{r})} \|f\|_{L^{r'}_x(\R^d)}
\end{equation}
for any $2 \leq r \leq \infty$ (where of course $r' := r/(r-1)$), together with H\"older's inequality, we can thus bound \eqref{ok} as
$$ \eqref{ok} \lesssim \int_{R^2}^T t^{-d(\frac{1}{2}-\frac{1}{r})}
\| \eta_{R/10} F(u(t)) \|_{L^{r'}_x(\R^d)} \| \chi_R f \|_{L^{r'}_x(\R^d)}\ dt.$$
By H\"older's inequality again we have $\| \chi_R f \|_{L^{r'}_x(\R^d)} \lesssim R^{d(\frac{1}{2}-\frac{1}{r})}$.  If we now choose $r$ so that $r' = 2/p$, then from \eqref{xr3} we have
\begin{equation}\label{etarf}
\| \eta_{R/10} F(u(t)) \|_{L^{r'}_x(\R^d)} \lesssim R^{-p\alpha/2} \lesssim R^{-\alpha/2}.
\end{equation}
The hypothesis $p > 1 + \frac{4}{d}$ implies that $d(\frac{1}{2}-\frac{1}{r}) > 2 \geq 1$, and so we conclude that
$$ \eqref{ok} \lesssim R^{-d(\frac{1}{2}-\frac{1}{r})+2} R^{-\alpha/2} $$
and thus the contribution of \eqref{ok} to \eqref{charge} is $O( R^{-\alpha - \delta} )$ for some $\delta > 0$, which is acceptable.  

Now we turn to \eqref{ok2}.  We square this expression to obtain
$$ \eqref{ok2}^2 = \int_0^{R^2} \int_0^{R^2} \langle e^{i(t'-t)\Delta} \eta_{R/10} F(u(t)), \eta_{R/10} F(u(t')) \rangle_{L^2}\ dt dt'.$$   
We could invoke the dispersive inequality \eqref{eitf} again to control this expression, but this turns out to give inferior results.  To get better results, we exploit the spherical symmetry.  In general, the fundamental solution of the free Schr\"odinger propagator gives that
$$ |\langle e^{it\Delta} f, g \rangle_{L^2}| = \frac{1}{(4\pi|t|)^{d/2}} 
\left|\int_{\R^d} \int_{\R^d} e^{i|x-y|^2/4t} f(x) \overline{g(y)}\ dx dy\right|$$
but when $f$ and $g$ are spherically symmetric, we can average over the angular variable and obtain
$$ |\langle e^{it\Delta} f, g \rangle_{L^2}| = \frac{1}{(4\pi|t|)^{d/2}} 
\left|\int_{\R^d} \int_{\R^d} (\int_{S^{d-1}} e^{i||x|\omega-|y||^2/4t}\ d\omega) f(x) \overline{g(y)}\ dx dy\right|.$$
A standard stationary phase computation reveals that
$$ \int_{S^{d-1}} e^{i||x|\omega-|y||^2/4t}\ d\omega \lesssim (|x| |y| / |t|)^{-(d-1)/2}$$
and so we have the spherically symmetric dispersive inequality
$$ |\langle e^{it\Delta} \eta_{R/10} f, \eta_{R/10} g \rangle_{L^2}| \lesssim \frac{1}{|t|^{1/2} R^{d-1}} \|f\|_{L^1_x(\R^d)} \|g\|_{L^1_x(\R^d)}$$
which improves over \eqref{eitf} when $|t| \leq R^2$. On the other hand, from Cauchy-Schwarz we have
$$ |\langle e^{it\Delta} \eta_{R/10} f, \eta_{R/10} g \rangle_{L^2}| \leq \|f\|_2 \|g\|_2$$
and hence by interpolation
$$ |\langle e^{it\Delta} \eta_{R/10} f, \eta_{R/10} g \rangle_{L^2}| \lesssim \frac{1}{|t|^{\frac{1}{2}-\frac{1}{r}} R^{(d-1)(1-\frac{2}{r})}} \|f\|_{L^{r'}_x(\R^d)} \|g\|_{L^{r'}_x(\R^d)}$$
for any $2 \leq r \leq \infty$.  Setting $r' := 2/p$ and using \eqref{etarf}, we thus have
$$ \eqref{ok2}^2 \lesssim \int_0^{R^2} \int_0^{R^2} \frac{1}{|t-t'|^{\frac{1}{2}-\frac{1}{r}} R^{(d-1)(1-\frac{2}{r})}} R^{-\alpha}\ dt dt'.$$
The right-hand side is $O( R^{-d(1-\frac{2}{r})+4} R^{-\alpha})$, and so the contribution of \eqref{ok2} to \eqref{charge} is also $O( R^{-\alpha - \delta} )$ for some $\delta > 0$.  Thus the net contribution of $w_+$ is acceptable.  By symmetry we see that the contribution of $w_-$ is also acceptable.

To finish the proof of \eqref{chir} and hence the proposition, it suffices to show that
$$ |\langle v_+, v_- \rangle_{L^2}| \lesssim R^{4-d} + R^{-\alpha - \delta}.$$
We expand the left-hand side as
$$
|\int_0^T \int_{-T}^0 \langle (1-\eta_{R/10}) e^{-i t_- \Delta} \chi_R^2 e^{-it_+ \Delta} (1-\eta_{R/10}) F(u(t_+)), F(u(t_-)) \rangle_{L^2}\ dt_+ dt_-|.$$
Now let us inspect the integral kernel $K_{t_-,t_+}(x,y)$ of $e^{i t_- \Delta} \chi_R^2 e^{it_+ \Delta}$, which (thanks to the fundamental solution for the free Schr\"odinger propagator) is given by the formula
$$ K_{t_-,t_+}(x,y) = \frac{C R^d}{|t_+|^{d/2} |t_-|^{d/2}} \int_{\R^d} e^{-i \Phi_{x,y}(z)} \chi^2(z)\ dz$$
for some absolute constant $C$, where the quadratic phase $\Phi_{x,y}(z)$ is given by the formula
$$ \Phi_{x,y}(z) = \frac{|y-Rz|^2}{4t_+} + \frac{|Rz-x|^2}{4t_-}.$$
Because of the cutoffs $1-\eta_{R/10}$, we are only interested in this kernel in the regime when $|x|, |y| \leq R/10$.  Meanwhile, we have $1/2 \leq |z| \leq 1$ on the support of $\chi$.  Since $t_-, t_+$ are positive, we conclude that the gradient
$$ \nabla_z \Phi_{x,y}(z) = R \frac{Rz-y}{4t_+} + R \frac{Rz-x}{4t_-}$$
has magnitude bounded away from zero by $\gg R^2 / \min(|t_-|,|t_+|)$.  We can thus integrate by parts repeatedly and obtain the bound
$$ |K_{t_-,t_+}(x,y)| \lesssim \frac{R^d}{|t_+|^{d/2} |t_-|^{d/2}} \min\left( 1, (\frac{\min(|t_-|,|t_+|)}{R^2})^{100d}\right).$$
We thus have
\begin{align*}
  |\langle v_+, v_- \rangle_{L^2}| 
&\lesssim \int_0^T \int_{-T}^0 \frac{R^d}{|t_+|^{d/2} |t_-|^{d/2}} \min\left( 1, (\frac{\min(|t_-|,|t_+|)}{R^2})^{100d}\right) \\
&\quad \| (1-\eta_{R/10}) F(u(t_+))\|_{L^1_x(\R^d)}
\| (1-\eta_{R/10}) F(u(t_-))\|_{L^1_x(\R^d)}\ dt_+ dt_-.
\end{align*}
A direct computation shows that
$$ \int_0^T \int_{-T}^0 \frac{R^d}{|t_+|^{d/2} |t_-|^{d/2}} \min\left( 1, (\frac{\min(|t_-|,|t_+|)}{R^2})^{100d}\right)\ dt_- dt_+  \lesssim R^{4-d}$$
and so it suffices to show that
$$ \| (1-\eta_{R/10}) F(u(t))\|_{L^1_x(\R^d)} \lesssim 1 + R^{(d-4-\alpha-\delta)/2}$$
for all $t$.

Fix $t$.  The contribution of $Vu$ to this expression is $O(1)$ thanks to \eqref{supt}, so it suffices to show that
$$ \int_{|x| \leq R/10} |u(t,x)|^p\ dx \lesssim 1 + R^{(d-4-\alpha-\delta)/2}.$$
If $p \geq 2$, then this follows from Sobolev embedding and \eqref{supt} (since $p < 1 + \frac{4}{d-2}$), so suppose that $p < 2$.  But then from \eqref{xr3} and H\"older's inequality we have
$$
\int_{|x| \sim R'} |u(t,x)|^p\ dx \lesssim (R')^{-\alpha p/2} (R')^{-d(\frac{p}{2}-1)}.$$
for any $R' \geq 1$.  Summing this over dyadic $R'$ between $1$ and $R$ (and using one last application of H\"older to treat the case $|x| = O(1)$) we conclude that
$$
\int_{|x| \leq R/10} |u(t,x)|^p\ dx \lesssim 1 + R^{-\alpha p/2} R^{-d(\frac{p}{2}-1)} \log R.$$
Since $p > 1+\frac{4}{d}$ and $\alpha p \geq \alpha$, the claim follows.  The proof of Proposition \ref{decay-prop} is now complete.
\end{proof}

\section{Virial inequalities}\label{virial-sec}

We now return to the proof of Theorem \ref{quasi-thm}.  Let $u$ be a spherically symmetric almost periodic global solution obeying the energy bound \eqref{supt} for some large $E$ (it will be important to track which bounds are uniform in $E$).  From \eqref{virial} and the fundamental theorem of calculus we see that
\begin{equation}\label{virial2}
\begin{split}
2 \int_{T_1}^{T_2} \int_{\R^d} \operatorname{Hess}(a)( \nabla u, \overline{\nabla u} )\ dx dt&\\
+ \frac{p-1}{p+1} \int_{T_1}^{T_2}\int_{\R^d} |u|^{p+1} \Delta a\ dx dt& \\
- \frac{1}{2} \int_{T_1}^{T_2} \int_{\R^d} |u|^2 \Delta \Delta a\ dx dt& \\
- \int_{T_1}^{T_2} \int_{\R^d} (\nabla a \cdot \nabla V) |u|^2\ dx dt& \lesssim \sup_{T_1 \leq t \leq T_2} \int_{\R^d} |u| |\nabla u| |\nabla a|\ dx
\end{split}
\end{equation}
for all times $-\infty < T_1 < T_2 < +\infty$ and all test functions $a \in C^\infty_0(\R^d)$.  From the bounded energy of $u$, the dominated convergence theorem, and a standard truncation argument, we see that this inequality also holds for any smooth $a$ with $\nabla a, \nabla^2 a, \nabla^3 a, \nabla^4 a$ uniformly bounded.

Suppose that we are in dimension $d \geq 7$.  Then from Proposition \ref{decay-prop} we have
$$ \int_{\R^d} |u|^2 |x|^2\ dx \lesssim_E 1$$
and thus by \eqref{supt} and Cauchy-Schwarz we have
$$ \int_{\R^d} |u| |\nabla u| |x|\ dx \lesssim_E 1.$$
If we formally apply \eqref{virial2} with $a(x) := |x|^2$, we obtain the virial inequality
\begin{equation}\label{virial3}
\begin{split}
4 \int_{T_1}^{T_2} \int_{\R^d} |\nabla u|^2\ dx dt&\\
+ 2d \frac{p-1}{p+1} \int_{T_1}^{T_2}\int_{\R^d} |u|^{p+1}\ dx dt& \\
- \int_{T_1}^{T_2} \int_{|x| \lesssim 1} O(|u|^2)\ dx dt& \lesssim_E 1.
\end{split}
\end{equation}
One can justify \eqref{virial3} rigorously as follows.  We let $R \gg 1$ be a large radius, and apply \eqref{virial2} with $a$ chosen to equal $|x|^2$ for $|x| \leq R$, vanishing for $|x| \geq 2R$, and smoothly interpolated in between.  The terms coming from the region $|x| \geq R$ to go to zero as $R \to \infty$ (keeping $T_1, T_2$ fixed) by \eqref{supt} and the dominated convergence theorem, yielding \eqref{virial3} by monotone convergence.  

Using conservation of energy \eqref{energy-def} and H\"older's inequality we can bound
$$
4 \int_{\R^d} |\nabla u|^2\ dx 
+ 2d \frac{p-1}{p+1} \int_{\R^d} |u|^{p+1}\ dx dt
- \int_{|x| \lesssim 1} O(|u|^2)\ dx dt \geq c \E(u) - O(1)$$
for some absolute constant $c>0$ depending only on $p$ and $d$.  From \eqref{virial3}, we conclude that
$$ \E(u) \lesssim 1 + O_E( \frac{1}{T_2-T_1} ).$$
Letting $T_2 - T_1 \to \infty$, we conclude that the energy $E(u)$ of the almost periodic solution is bounded uniformly in $E$:
\begin{equation}\label{energyu}
 \E(u) \lesssim 1.
\end{equation}

Having controlled the energy, we now turn to the mass.  Formally, the idea is to apply \eqref{virial2} with $a(x) := |x|^4$.  If we are in dimension $d \geq 11$, then from Proposition \ref{decay-prop} we have
\begin{equation}\label{u6}
 \int_{\R^d} |u|^2 |x|^6\ dx \lesssim_E 1
\end{equation}
and thus by \eqref{supt} (or \eqref{energyu}) and Cauchy-Schwarz we have
\begin{equation}\label{n3}
\int_{\R^d} |u| |\nabla u| |x|^3\ dx \lesssim_E 1.
\end{equation}
One can also compute
\begin{align*}
\operatorname{Hess}(a)( \nabla u, \overline{\nabla u} ) &= 12 |x|^{2}  |u_r|^2 + 4 |x|^{2} |\nabb u|^2 \\
\Delta a &= 4 (d+2) |x|^2 \\
\Delta \Delta a &= 8d (d+2)
\end{align*}
where $u_r$ is the radial derivative, and thus we formally have
\begin{equation}\label{virial4}
\begin{split}
\int_{T_1}^{T_2} \int_{\R^d} 24 |x|^2 |u_r|^2 + 8 |x|^2 |\nabb u|^2\ dx dt&\\
+ 4 (d+2) \frac{p-1}{p+1} \int_{T_1}^{T_2}\int_{\R^d} |u|^{p+1} |x|^2\ dx dt& \\
- 4d(d+2) \int_{T_1}^{T_2} \int_{\R^d} |u|^2\ dx dt& \\
- \int_{T_1}^{T_2} \int_{\R^d} (\nabla a \cdot \nabla V) |u|^2\ dx dt& \lesssim_E 1.
\end{split}
\end{equation}
To justify \eqref{virial4} rigorously, we take a large $R > 1$ and apply \eqref{virial2} with $a$ equal to $|x|^4$ for $|x| < R$, equal to $100 R^3 |x|$ for $|x| \geq 2R$, and smoothly interpolated in between in such a way that $a$ remains convex (so in particular $\hbox{Hess}(a)$ is positive semi-definite and $\Delta a$ is non-negative).   The terms coming from the region $R \leq |x| \leq 2R$ either goes to zero as $R \to \infty$ (thanks to \eqref{u6}), or are non-negative, and one can easily deduce \eqref{virial4} by monotone convergence.

Using \eqref{energyu} we can estimate
$$ \int_{T_1}^{T_2} \int_{\R^d} (\nabla a \cdot \nabla V) |u|^2\ dx dt = O( T_2 - T_1 ).$$
We discard the positive terms $|x|^2 |\nabb u|^2$ and $|u|^{p+1} |x|^2$ and end up with
\begin{equation}\label{ttt}
 \int_{T_1}^{T_2} \int_{\R^d} 24 |x|^2 |u_r|^2 - 4d(d+2) |u|^2\ dx dt \lesssim T_2 - T_1 + O_E(1).
\end{equation}
To deal with the negative term $4d(d+2) |u|^2$ we use

\begin{lemma}[Hardy's inequality]  Suppose that $f \in C^\infty_0(\R^d)$ vanishes near the origin, and let $\beta \in \R$.  Then
$$ \left(\frac{d+\beta}{2}\right)^2 \int_{\R^d} |f(x)|^2 |x|^\beta\ dx \leq \int_{\R^d} |f_r(x)|^2 |x|^{\beta+2}\ dx.$$
\end{lemma}

\begin{proof} Start with the trivial inequality
$$ \int_{\R^d} \left||x| f_r(x) + \frac{d+\beta}{2} f(x)\right|^2 |x|^\beta\ dx \geq 0$$
and rearrange the left-hand side by integration by parts.
\end{proof}

Applying this with $\beta = 0$ and $f$ equal to a smoothly truncated version of $u$ (both near zero and near infinity) and applying a limiting argument (using \eqref{u6} and \eqref{supt} to control errors), we conclude that
$$ \frac{d^2}{4} \int_{\R^d} |u|^2\ dx \leq \int_{\R^d} |u_r|^2 |x|^2\ dx.$$
Multiplying this by $24$ and inserting it into \eqref{ttt}, we conclude that
$$ 2d(d-4) \int_{T_1}^{T_2} \int_{\R^d} |u|^2\ dx dt \lesssim T_2 - T_1 + O_E(1).$$
Note that as $d \geq 5$, the constant on the left-hand side is positive\footnote{More generally, it turns out that one can always use Hardy's inequality to obtain a favourable sign in this manner in dimensions $d \geq 5$ when selecting any weight of the form $a(x) = |x|^{\alpha}$ for some $\alpha \geq 1$.  Of course, for large $\alpha$ one still needs to establish sufficiently strong spatial decay of $u$ and/or $\nabla u$ in order to ensure that the right-hand side of \eqref{virial2} remains finite and to rigorously justify the use of this non-compactly supported weight.}.
Using conservation of mass \eqref{mass-def} and letting $T_2 - T_1 \to \infty$ as before we conclude that
\begin{equation}\label{massu}
 \M(u) \lesssim 1;
 \end{equation}
combining this with \eqref{energyu} we conclude that
\begin{equation}\label{outh}
 \sup_t \|u(t)\|_H \lesssim 1.
\end{equation}
Invoking Proposition \ref{decay-prop} again we conclude that
\begin{equation}\label{xdecay}
 \sup_t \int_{|x| \geq R} |u(t,x)|^2\ dx \lesssim R^{4-d}
\end{equation}
for all $R \geq 1$.

Finally, the arguments used to prove\footnote{As already noted in Section \ref{quasisec}, the addition of the potential energy term $Vu$ does not impact the proof of this proposition.  Also, the proposition as stated only controls $u(t)$ for sufficiently large times $t$, in order to obtain decay of the linear solutions $e^{it\Delta} u_0$, $e^{it\Delta} u_+$, but for almost periodic solutions, the contribution of the linear solutions can always be neglected using Riemann-Lebesgue type lemmas, and so the estimates for almost periodic solutions are valid for all times.} \cite[Proposition 6.1]{attractor}, when combined with \eqref{outh}, show that
$$ \sup_t \| u(t) \|_{H^{1+\eta}_x(\R^d)} \lesssim 1$$
for some $\eta > 0$ depending only on $p,d$.  Combining this with \eqref{xdecay} we see that $u(t)$ lies inside a compact subset $K$ of $H$ that depends only on $p,d,V$ (cf. \cite[Proposition B.1]{attractor}).  The proof of Theorem \ref{quasi-thm} and thus Theorem \ref{main-thm} is thus complete.

\section{Remarks and possible generalisations}\label{remarks-sec}

The hypothesis $d \geq 11$ can be improved with some additional work.  First of all, one can exploit the fact that the expression $\int_{\R^d} \nabla a \cdot \Im( \overline{u} \nabla u )\ dx$ appearing in \eqref{virial} is itself a derivative of a usable expression:
$$ \int_{\R^d} \nabla a \cdot \Im( \overline{u} \nabla u )\ dx = \frac{d}{dt} \frac{1}{2} \int_{\R^d} a |u|^2\ dx.$$
Thus to bound the left-hand side on average in time, it suffices to control $\int_{\R^d} a |u|^2\ dx$ uniformly in time.  Because of this, we can weaken \eqref{u6} to
$$  \int_{\R^d} |u|^2 |x|^4\ dx \lesssim_E 1$$
and still continue the rest of the proof.  This lets us relax the condition $d \geq 11$ to $d \geq 9$.

One can do even better by establishing an analogue of the decay result in Proposition \ref{decay-prop} for the derivative $\nabla u$.  Indeed, by repeating the proof of that proposition (but using the nonlinearity $\nabla F(u)$ rather than $F(u)$) one should eventually establish the bound
$$  \int_{|x| \geq R} |\nabla u(t,x)|^2\ dx \lesssim_E R^{2-d}$$
for $R \geq 1$. Note that the bound for $\nabla u$ is actually better than that for $u$; one expects $\nabla u$ to decay one order of magnitude faster than $u$ (as one can already heuristically see by looking at the ground state equation $Q = - (-\Delta+E)^{-1} ( VQ + |Q|^{p-1} Q )$ and considering the regularity of the resolvent kernel $(-\Delta+E)^{-1}$). See also \cite{tao:compact} for another instance of this phenomeon.  This allows one to establish \eqref{n3} for all $d \geq 7$, and should also let one extend Theorem \ref{main-thm} to this case.

The additional truncations to the virial identity in \cite[Sections 9,10]{tao:compact} should also allow one to derive \eqref{energyu} for all $d \geq 5$ (and perhaps even $d \geq 3$), but the author was not able to adapt the same argument to prove \eqref{massu} for $d=5$ or $d=6$.  Nevertheless the author believes that Theorem \ref{quasi-thm} (and thus Theorem \ref{main-thm}) should hold for all $d \geq 5$.

The requirement that $V$ be compactly supported can easily be relaxed to some polynomial decay rate on $V$ and $\nabla V$, though we have not attempted to compute the optimal such rate\footnote{A back-of-the-envelope computation suggests that one needs $|V(x)| \leq c |x|^{-2}$ and $|\nabla V(x)| \leq c |x|^{-3}$ for some small absolute constant $c>0$ and all sufficiently large $x$.}.  Note that some regularity on $V$ is required in order to keep the equation \eqref{nls} well-posed in the energy class $H$.

Our arguments rely at several key junctures on spherical symmetry.  In the absence of spherical symmetry, the problem is now translation invariant, and one must modify the notion of a compact attractor to take this into account; see \cite{attractor}.  Nevertheless, since the potential $V$ can only counteract the defocusing nonlinearity near the origin, it is reasonable to expect that some counterpart of Theorem \ref{quasi-thm} and Theorem \ref{main-thm} holds in this setting, at least in sufficiently high dimension.  One possible initial step in this direction would be to remove the assumption of spherical symmetry from Proposition \ref{decay-prop}.  One may also wish to apply interaction virial estimates (as in \cite{ckstt:scatter}) in this case.

One might also wish to consider a model in which the attraction is caused by a locally focusing nonlinearity rather than by a potential term.  For instance, one could consider the equation $iu_t + \Delta u = F(u)$ where $F$ is a Hamiltonian nonlinearity which behaves like the defocusing nonlinearity $|u|^{p-1} u$ for large $u$ but is allowed to be negative for small $u$.  It is certainly possible for such models to admit non-trivial nonlinear bound states.  However it is not clear to the author whether the analogue of Theorem \ref{quasi-thm} or Theorem \ref{main-thm} holds in this setting, even in extremely high dimension.  It is not even clear that the space of spherically symmetric nonlinear bound states is bounded in the energy space.  One possible obstruction arises from the fact that one can build partly bound states by starting with a nonlinear bound state in a lower dimension and extending it trivially to higher dimensions.  Such bound states have infinite energy, but one could imagine that some truncation or perturbation of this partly bound state would be stable, leading to nonlinear bound states or other soliton-like solutions to this equation of arbitrarily large but finite mass and energy.  As a variant of this scenario, one could consider spherically symmetric solutions concentrated around an annulus $\{ x: |x| = R + O(1) \}$ for some large $R$; in polar coordinates, such solutions resemble a one-dimensional nonlinear bound state, and by varying the parameter $R$ this could conceivably create a family of nonlinear bound states or similar solutions of arbitrarily large mass and energy.  On the other hand, it may well be possible to show that the $L^\infty$ norm of almost periodic solutions to such equations are necessarily bounded by some absolute constant depending only on the dimension and the nonlinearity (here it may be convenient to add the additional assumption that the nonlinearity $F$ is smooth, as this should force the almost periodic solution to be smooth also).

\end{document}